\documentclass[12pt, a4paper]{amsart}
\let\cal\mathcal
\usepackage{amssymb}
\usepackage{amsthm}
\usepackage{amsfonts}
\usepackage{url}
\newtheorem{lemma}{\bf Lemma}[section]

\newtheorem{proposition}[lemma]{\bf Proposition}
\newtheorem{conjecture}[lemma]{\bf Conjecture}
\begin{document}
\title{Hadwiger's conjecture for hypergraphs}
\begin{abstract}
In 1943, Hadwiger formulated his celebrated conjecture \cite{hadw},
	connecting the chromatic number $\chi(G)$ of a finite, 
	simple, undirected graph with the cardinality of the largest
	complete minor, $\eta(G)$. The disjoint union of all finite
	complete graphs shows that Hadwiger's conjecture fails for infinite
	graphs \cite{zyp1}, but a slightly weaker version is true
	in these graphs \cite{zyp2}, and open for finite graphs. In this
	note we generalize that weaker version to hypergraphs and provide
	a simple, general, and purely set-theoretical formulation of Hadwiger's 
	conjecture.
\end{abstract}
\author{Dominic van der Zypen}
\address{Federal office of border and customs security, Taubenstrasse 16,
CH-3003 Bern, Switzerland}
\email{dominic.zypen@gmail.com}
\maketitle
\section{Introduction}
The conjecture of Hadwiger, formulated in 1943 \cite{hadw}, is one of the central open
problems in graph theory. Recall that a graph $F$ is a {\em minor} of a graph $G$
if a graph isomorphic to $F$ can be obtained from $G$ by a sequence
of vertex deletions, edge deletions, and edge contractions.
\begin{conjecture} {\em (Hadwiger's conjecture.)} Let $G$ be a finite graph 
	and $t\in \mathbb{N}$. If $\chi(G) > t$, then $G$ has $K_t$ as a minor.
	Equivalently, if $G$ does not have $K_t$ as a minor, then $\chi(G) < t$. 
\end{conjecture}
The disjoint union of all finite
complete graphs shows that Hadwiger's conjecture fails for infinite
graphs \cite{zyp1}, but a slightly weaker version is true
in these graphs \cite{zyp2}, and open for finite graphs. In this
note we generalize that weaker version to hypergraphs.
\section{Basic notions}
\subsection{Hypergraphs} A {\em hypergraph} $H=(V,E)$ consists of a
set $V$ and $E\subseteq {\mathcal P}(V)$, that is, $E$ 
consists of subsets of $V$ of arbitrary size. 

\subsection{Connected subsets} A subset $S\subseteq V$ of a hypergraph $H=(V,E)$ 
is {\em connected} if for every nonempty 
$X\subseteq S$ with $X\neq S$ there is $e\in E$ such that 
$$(e\cap S) \cap X \neq \emptyset\text{ and }
(e\cap S)\setminus X\neq \emptyset.$$

\subsection{Connected to each other} If $H=(V,E)$ is a hypergraph
and $S_1, S_2\subseteq V$ are disjoint, we say they are {\em connected
to each other} if there is $e\in E$ such that $$e\cap S_1 \neq \varnothing
\text{ and }  e\cap S_2\neq \varnothing.$$

\subsection{Colouring} Let $H=(V,E)$ be a hypergraph and 
$\kappa\neq \varnothing$ be a cardinal.
Then a map $c:V\to \kappa$ is said to be a {\em colouring} if for
every $e\in E$ with $|e|\geq 2$ we have that
the restriction $c\restriction_e$ is non-constant. The {\em chromatic
number} $\chi(H)$ of $H$ is the smallest cardinal $\kappa$ such there
is a colouring $c:V\to \kappa$.

\section{A form of Hadwiger's conjecture for hypergraphs}

Assume that $H=(V,E)$ is a hypergraph and let's assume 
$V\neq \varnothing \neq E$ to avoid pathologies. Moreover, for
any set $X$, we set $[X]^2= \big\{\{x,y\}:x, y\in X \land x\neq y\big\}$. 

The {\sl
weak Hadwiger conjecture for hypergraphs} states that:
\begin{quote}
	(WHC): Suppose $\kappa\neq\emptyset$ is a cardinal such that for every map $c:V\to\kappa$
	there is $e\in E$ with $|e|\geq 2$ and $c|_e:e\to \kappa$ is constant. Then
	there is ${\cal S}\subseteq {\cal P}(V)$ such that 
        \begin{enumerate}
		\item $|{\cal S}| =\kappa$,
		\item every member of ${\cal S}$ is connected, and
		\item whenever $T_0, T_1\in {\cal S}$ with $T_0\neq T_1$, then
			$T_0, T_1$ are connected to each other.
	\end{enumerate}
\end{quote}

\begin{proposition}\label{main} {\em
	(WHC) is equivalent to (WHC) restricted to graphs (= hypergraphs
	in which every edge has size $2$).}
\end{proposition}
{\em Proof.} One direction is evident. So we assume that (WHC) restricted to graphs hold.
Let $H=(V, E)$ be a hypergraph and 
	suppose $\kappa\neq\emptyset$ is a cardinal such that for every map $c:V\to\kappa$
	there is $e\in E$ with $|e|\geq 2$ and $c|_e:e\to \kappa$ is constant.
	
Consider the graph $G(H) :=(V, E_2)$ where $$E_2 = \bigcup \{[e]^2: e\in E\}.$$

Let $c:V\to \kappa$ be any map. We know that there is $e_0\in E$ with $|e_0|\geq 2$ and 
$c|_{e_0}$ is constant. Pick any distinct $x,y\in e_0$. So $\{x,y\}\in E_2$ and $c|_{\{x,y\}}$
is constant. Since $c: V\to \kappa$ was arbitrary, we can apply (WHC) restricted to graphs
to the graph $G(H) = (V, E_2)$ and get a set ${\cal S}\subseteq {\cal P}(V)$ of cardinality
$\kappa$ and obeying the three requirements noted in the definition of (WHC) above.

Finally, it is easy to check that every member of ${\cal S}$ is connected in the hypergraph
$H=(V, E)$, and every two distinct members of ${\cal S}$ are connected to each other in that 
hypergraph.\hfill{$\Box$}

\section{Concluding remarks}
The simple result contained in \ref{main} does not mean that the study of
hypergraphs and their colorings and minors is not worthwhile. 
Even though the general conjecture is logically equivalent 
for graphs and hypergraphs, \cite{stei} shows that hypergraph colorings 
and minors are an interesting subject in their own right.


\end{document}